\documentclass[11pt]{article}
\usepackage{mathrsfs}
\usepackage{bbm}

\usepackage{verbatim}

\usepackage{amsthm}
\usepackage{amsmath}
\usepackage{amssymb}

\usepackage{graphicx}
\usepackage{epstopdf}
\usepackage{titling}

\usepackage{enumerate}

\usepackage{algpseudocode}
\usepackage{algorithm}
\usepackage{algorithmicx}
\usepackage{tikz}
\usepackage{appendix}

\def\F{{\mathcal F}}
\newtheorem{theorem}{Theorem}[section]

\newtheorem{lemma}[theorem]{Lemma}

\newtheorem{corollary}[theorem]{Corollary}

\title{Maximizing the number of nonnegative subsets}

\author{Noga Alon
\thanks{
Sackler School of Mathematics and Blavatnik School of Computer
Science,
Tel Aviv University, Tel Aviv 69978, Israel and
Institute for Advanced Study,
Princeton, New Jersey, 08540, USA.
Email: {\tt nogaa@tau.ac.il}.
Research supported in part by an ERC Advanced grant,
by a USA-Israeli BSF
grant, by an ISF grant, by the Israeli I-Core program and
by the Simonyi Fund.
}
\and Harout Aydinian \thanks{Department of Mathematics, University of Bielefeld, Germany. Email: {\tt ayd@math.uni-bielefeld.de}}
\and Hao Huang\thanks{Institute for Advanced Study, Princeton, NJ 08540
and DIMACS at Rutgers University. Email: {\tt huanghao@math.ias.edu}.
Research supported in part by the IAS-DIMACS postdoctoral fellowship.}}

\begin{document}
\maketitle
\begin{abstract}
Given a set of $n$ real numbers, if the sum of elements of every subset of
size larger than $k$ is negative, what is the maximum number of subsets
of nonnegative sum? In this note we show that the answer is
$\binom{n-1}{k-1} + \binom{n-1}{k-2} + \cdots + \binom{n-1}{0}+1$,
settling a problem of Tsukerman.
We provide two proofs, the first
establishes and applies a weighted version of Hall's Theorem and
the second is based on an extension of the nonuniform Erd\H{o}s-Ko-Rado
Theorem.
\end{abstract}

\section{Introduction}

Let $\{x_1, \cdots, x_n\}$ be a sequence of $n$ real numbers whose
sum is negative. It is natural to ask the following question: What
is the maximum possible number
of subsets of nonnegative sum it can have? One can set $x_1=n-2$
and $x_2 = \cdots = x_n=-1$. This gives $\sum_{i=1}^n x_i = -1 < 0$ and
$2^{n-1}$ nonnegative subsets, since all the proper subsets containing
$x_1$, together with the empty set, have a nonnegative sum. It is also
not hard to see that this is best possible, since for
every subset $A$, either $A$ or its complement $\{x_1, \cdots, x_n\}
\backslash A$ must have a negative sum. Now a new question arises: suppose
it is known that every subset of size larger than $k$ has a negative sum,
what is the maximum number of nonnegative subsets? This question
was raised recently by Emmanuel Tsukerman \cite{Ts}.
The previous problem
is the special case when $k=n-1$.  A similar construction $x_1=k-1$,
$x_2=\cdots=x_n=-1$ yields a lower bound $\binom{n-1}{k-1} + \cdots +
\binom{n-1}{0} + 1$. In this note we prove that this is also tight.
\begin{theorem}
\label{main_theorem}
Suppose that every subset of $\{x_1, \cdots, x_n\}$ of size larger
than $k$ has a negative sum, then there are at most
$\binom{n-1}{k-1} + \cdots + \binom{n-1}{0}+1$
subsets with nonnegative sums.
\end{theorem}
One can further ask whether the extremal configuration $x_1=k-1, x_2=
\cdots = x_n=-1$ is unique, in the sense that the family $\mathcal{F}=\{U:
\sum_{i \in U} x_i \ge 0\}$ is unique up to isomorphism. Note that when
$k=n-1$, an alternative construction $x_1 = -n$, $x_2 = \cdots, x_n=1$
also gives $2^{n-1}$ nonnegative subsets, while the family $\mathcal{F}$
it defines is non-isomorphic to the previous one. More generally,
for $k=n-1$ any sequence $X=\{x_1, \ldots ,x_n\}$ of $n$
integers whose sum is $-1$ contains
exactly $2^{n-1}$ nonnegative subsets, as for any subset $A$ of
$X$,
exactly one of the two sets $A$ and $X-A$ has a nonnegative sum.
However, for every
$k<n-1$, we can prove the uniqueness by the following result in which the
number of nonnegative elements in the set is also taken into account.
\begin{theorem}
\label{strengthened_theorem}
Let $1 \leq t \le k < n$ be integers, and let $X$ be
a set of real numbers $\{x_1, \cdots,
x_n\}$, in which there are exactly $t$ nonnegative numbers. Suppose
that the sum of elements of every subset of size greater than $k$
is negative, then the number of nonnegative subsets is at most
$2^{t-1}(\binom{n-t}{k-t} + \cdots + \binom{n-t}{0}+1)$. This is
tight for all admissible values of $t,k$ and $n$.
\end{theorem}

For every fixed $k$ and $n$ with $k< n-1$, the expression in Theorem
\ref{strengthened_theorem} is strictly decreasing in $t$.
Indeed, if $1 \leq t < t+1 \leq k \leq n$, then, using Pascal's
identity:
$$
2^{t-1}\left(\binom{n-t}{k-t} + \cdots + \binom{n-t}{0}+1\right)
-
2^{t}\left(\binom{n-t-1}{k-t-1} + \cdots + \binom{n-t-1}{0}+1\right)
$$
$$
=2^{t-1} \left[ \binom{n-t}{k-t} + \cdots + \binom{n-t}{0}
- \binom{n-t-1}{k-t-1} - \cdots - \binom{n-t-1}{0}
- \binom{n-t-1}{k-t-1} - \cdots - \binom{n-t-1}{0}-1 \right]
$$
$$
=2^{t-1} \left[\binom{n-t}{k-t}-\binom{n-t-1}{k-t-1}+\binom{n-t}{0}
-\binom{n-t-1}{0}-1 \right]
$$
$$
=2^{t-1}\left[\binom{n-t-1}{k-t}-1 \right].
$$
The last quantity is strictly positive for all $t < k <n-1$ (and is
zero if $k=n-1$).

Therefore, the above theorem
implies Theorem \ref{main_theorem} as a corollary and shows that it is
tight for $k<n-1$
only when there is exactly one nonnegative number. The bound is
Theorem \ref{strengthened_theorem} is also tight by taking $x_1 = k-t$,
$x_2 = \cdots = x_t= 0$, $x_{t+1} = \cdots = x_n = -1$. In this example,
the sum of any $k+1$ elements is negative, and a subset is nonnegative
if and only if it is either of the form $\{x_1\} \cup S \cup T$, where $S$
is an arbitrary subset of $\{x_2, \cdots, x_t\}$ and $T$ is a subset of
$\{x_{t+1}, \cdots, x_n\}$ having size at most $k-t$, or when it is a
subset of $\{x_2, \cdots, x_t\}$.

The rest of this short paper is organized as follows. In Section
\ref{section_main} we prove a Hall-type theorem and deduce from it
the existence of perfect matchings in certain bipartite graphs.
This enables us to obtain Theorem \ref{strengthened_theorem} as a corollary.
Section \ref{section_alternative_proof} includes a
strengthening of the non-uniform version
of the Erd\H os-Ko-Rado theorem, which
leads to an alternative proof of Theorem \ref{main_theorem}.
In the last section, we discuss some further research
directions.

\section{The Main result}\label{section_main}

The following lemma  can be regarded as a
strengthening of the fundamental theorem of Hall \cite{hall}.
\begin{lemma}
\label{weighted_hall}
In a bipartite graph $G$ with two parts $A$ and $B$, suppose there exist
partitions $A=A_1 \cup \cdots \cup A_k$ and $B=B_1 \cup \cdots \cup B_l$,
such that for every $i \in [k], j \in [l]$, in the induced bipartite
graph $G[A_i, B_j]$ all the vertices in $A_i$ have equal degrees and
all the vertices in $B_j$ have equal degrees too. Define an auxiliary
bipartite graph $H$ on the same vertex set, and replace every nonempty
$G[A_i, B_j]$ by a complete bipartite graph. Then $G$ contains a perfect
matching if and only if $H$ contains a perfect matching.
\end{lemma}
\begin{proof}
The ``only if'' part is obvious since $G$ is a subgraph of $H$. In
order to prove the ``if'' part, note first that if $H$ contains a
perfect matching, then $\sum_{i=1}^k |A_i| =\sum_{j=1}^l |B_j|$. We will
verify that the graph $G$ satisfies the conditions in Hall's Theorem:
for any subset $X \subset A$, its neighborhood has size $|N_G(X)| \ge
|X|$. Put $Y=N_G(X)$, and
$$
X_i= X \cap A_i,~~~~~Y_j= Y \cap B_j,
$$
and define two sequences of numbers
$\{x_i\}$, $\{y_j\}$ so that
$$
|X_i|=x_i |A_i|,~~~~~|Y_j|=y_j|B_j|.
$$
Consider the pairs $(i, j)$ such that $G[A_i,B_j]$
is nonempty. In this induced bipartite subgraph suppose every vertex
in $A_i$ has degree $d_1$, and every vertex in $B_j$ has degree
$d_2$.
Double counting the number of edges gives
$d_1 \cdot |A_i| = d_2 \cdot |B_j|.$
On the other hand, we also have
$d_1 \cdot |X_i| \le d_2 \cdot |Y_j|$,
since every vertex in $X_i$ has exactly $d_1$ neighbors in $Y_j$,
and every vertex in $Y_j$ has at most $d_2$ neighbors in $X_i$. Combining
these two inequalities, we have $y_j \ge x_i$ for every pair $(i, j)$
such that $G[A_i, B_j]$ is nonempty. We claim that these inequalities
imply that $|Y| \ge |X|$, i.e.
\begin{equation}\label{hall_condition}
\sum_{j=1}^l |B_j| y_j \ge \sum_{i=1}^k |A_i| x_i.
\end{equation}
To prove the claim
it suffices to find $d_{i, j} \ge 0$ defined on every pair $(i, j)$ with
nonempty $G[A_i, B_j]$, such that
$$
\sum_{i, j} d_{i, j} (y_j -x_i) =
\sum_{j=1}^l |B_j| y_j - \sum_{i=1}^k |A_i| x_i.
$$
In other words,
the conditions for Hall's Theorem would be satisfied if the following
system has a solution:
\begin{equation}\label{eqn}
\sum_i d_{i, j} = |B_j|;~~~~\sum_j d_{i, j} =|A_i|;~~~~ d_{i,j}
\ge 0;~~~~d_{i,j}=0~\textrm{if~}G[A_i, B_j]=\emptyset.
\end{equation}
The standard way to prove that there is a solution is by
considering an appropriate flow problem.
Construct a network with a source $s$, a sink $t$, and vertices $a_1,
\cdots, a_k$ and $b_1, \cdots, b_l$. The source $s$ is connected to every
$a_i$ with capacity $|A_i|$, and every $b_j$ is connected to the sink $t$
with capacity $|B_j|$. For every pair $(i, j)$, there is an edge from
$a_i$ to $b_j$. Its capacity is $+\infty$ if $G[A_i, B_j]$ is nonempty and
$0$ otherwise. Then \eqref{eqn} is feasible if and only if there exists a
flow of value $\sum_i |A_i|= \sum_j |B_j|$. Now we consider an arbitrary
cut in this network: $(s \cup \{a_i\}_{i \in U_1} \cup \{b_j\}_{j \in
U_2}, t \cup \{a_i\}_{i \in [k] \backslash U_1} \cup \{b_j\}_{j \in
[l] \backslash U_2})$. Its capacity is finite only when for every $i
\in U_1, j \in [l]\backslash U_2$, $G[A_i, B_j]$ is empty. Therefore
in the auxiliary graph $H$, if we take $Z = \cup_{i \in U_1} A_i$, then
the degree condition $|N_H(Z)| \ge |Z|$ implies that $\sum_{j \in U_2}
|B_j| \ge \sum_{i \in U_1} |A_i|$ and thus the capacity of this cut
is equal to
$$\sum_{i \in [k] \backslash U_1} |A_i| + \sum_{j \in U_2}
|B_j| \ge \sum_{i \in [k]\backslash U_1} |A_i| + \sum_{i \in U_1} |A_i|
= \sum_{i=1}^k |A_i|.
$$
Therefore the minimum cut in this network has
capacity at least $\sum_{i=1}^k |A_i|$, and there is a cut of
exactly this capacity, namely the cut consisting of all edges
emanating from the source $s$. By the max-flow min-cut theorem,
we obtain a maximum flow of the same size and this provides us with a
solution $d_{i, j}$ to \eqref{eqn}, which verifies the Hall's condition
\eqref{hall_condition} for the graph $G$.  \end{proof}

\noindent \textbf{Remark.} Lemma \ref{weighted_hall} can also be
reformulated in the following way: given $G$ with the properties stated,
define the reduced auxiliary graph $H'$  on the vertex set $A'
\cup B'$, where $A'=[k]$, $B'=[l]$, such that $i \in A'$ is adjacent to
$j \in B'$ if $G[A_i, B_j]$ is nonempty. If for every subset $X \subset
A'$, $\sum_{j \in N_{H'}(X)} |B_j| \ge \sum_{i \in X} |A_i|$, then $G$
has a perfect matching. For the case of partitioning $A$ and $B$ into
singletons, this is exactly Hall's Theorem.

\begin{corollary}
\label{perfect_matching}
For $m \ge r+1$, let $G$ be the bipartite graph with two parts $A$ and
$B$, such that both parts consist of subsets of $[m]$ of size between $1$
and $r$.  $S \in A$ is adjacent to $T \in B$ iff $S \cap T = \emptyset$
and $|S| + |T| \ge r+1$. Then $G$ has a perfect matching.
\end{corollary}
\begin{proof}

For $1 \le i \le r$, let $A_i=B_i= \binom{[m]}{i}$, i.e. all the
$i$-subsets of $[m]$. Let us consider the bipartite graph $G[A_i,
B_j]$ induced by $A_i \cup B_j$. Note that when $i+j \le r$ or $i+j>m$,
$G[A_i, B_j]$ is empty, while when $r+1 \le i+j \le \min\{2r, m\}$, every
vertex in $A_i$ has degree $\binom{m-i}{j}$ and every vertex in $B_j$
has degree $\binom{m-j}{i}$. Therefore by Lemma \ref{weighted_hall},
it suffices to check that the reduced auxiliary graph $H'$ satisfies
the conditions in the above remark. We discuss the following two cases.

First suppose $m \ge 2r$, note that in the reduced graph $H'$,
$A'=B'=[r]$, every vertex $i$ in $A'$ is adjacent to the vertices
$\{r+1-i, \cdots, r\}$ in $B'$. The only inequalities we need to verify
are: for every $1 \le t \le r$, $\sum_{j=r+1-t}^r |B_j| \ge \sum_{i=1}^t
|A_i|$. Note that
$$
\sum_{j=r+1-t}^r |B_j| = \sum_{i=1}^t
\binom{m}{r-t+i} \ge \sum_{i=1}^t \binom{m}{i}.
$$
The last inequality holds because the function $\binom{m}{k}$
is increasing in $k$ when $k \le m/2$.

Now we consider the case $r+1 \le m \le 2r-1$. In this case
every vertex
$i$ in $A'$
is adjacent to vertices from $r+1-i$ to $\min\{r, m-i\}$. More precisely,
if $1 \le i \le m-r$, then $i$ is adjacent to $\{r+1-i, \ldots ,r\}$
in $B'$, and if
$m-r+1 \le i \le r$, then $i$ is adjacent to $\{r+1-i \ldots
,m-i\}$ in $B'$. It
suffices to verify the conditions for $X=\{1, \cdots, t\}$ when $
t \le r$, and for $X=\{s, \cdots, t\}$ when $m-r \le s \le t \le r$. In
the first case $N_{H'}(X)=\{r+1-t, \cdots, r\}$, and the desired
inequality holds since
$$
\sum_{j=r+1-t}^r \binom{m}{j} = \sum_{i=1}^r \binom{m}{i}
- \sum_{i=1}^{r-t} \binom{m}{i} \ge  \sum_{i=1}^r \binom{m}{i} -
\sum_{i=t+1}^{r} \binom{m}{i}= \sum_{i=1}^t \binom{m}{i}.
$$

For the second case, $N_{H'}(X)=\{r+1-t, \cdots, m-s\}$, and since $m
\ge r+1$,
$$
\sum_{j=r+1-t}^{m-s} \binom{m}{j} = \sum_{i=s}^{m-r+t-1}
\binom{m}{i} \ge \sum_{i=s}^t \binom{m}{i}.
$$
This concludes the proof of the corollary.
\end{proof}

We are now ready to deduce Theorem \ref{strengthened_theorem} from
Corollary \ref{perfect_matching}.

\begin{proof} of Theorem \ref{strengthened_theorem}:
Without loss of generality, we may assume that $x_1 \ge x_2 \ge \cdots
\ge x_n$, and $x_1 + \cdots +x_{k+1}<0$. Suppose there are $t \le k$
nonnegative numbers, i.e.
$x_1 \ge \cdots \ge x_t \ge 0$ and $x_{t+1}, \cdots, x_n < 0$.
If $t=1$, then every nonempty subset of nonnegative
sum must contain $x_1$, which gives at most
$\binom{n-1}{k-1} + \cdots + \binom{n-1}{0} + 1$
nonnegative subsets in total, as needed.

Suppose $t \ge 2$. We first partition all the subsets of $\{1, \cdots,
t\}$ into $2^{t-1}$ pairs $(A_i, B_i)$, with the property that $A_i \cup
B_i = [t]$, $A_i \cap B_i = \emptyset$ and  $1 \in A_i$. This can be done
by pairing every subset with its complement. For every $i$, consider the
bipartite graph $G_i$ with vertex set $V_{i,1} \cup V_{i,2}$ such that
$V_{i, 1} = \{A_i \cup S: S \subset \{t+1, \cdots, n\}, |S| \le k-t\}$ and
$V_{i, 2} = \{B_i \cup S: S \subset \{t+1, \cdots, n\}, |S| \le k-t\}$.
Note that if a nonempty subset with index set $U$ has a nonnegative sum,
then $|U \cap \{t+1, \cdots, n\}| \le k-t$, otherwise $U \cup \{1, \cdots,
t\}$ gives a nonnegative subset with more than $k$ elements. Therefore
every nonnegative subset is a vertex of one of the graphs $G_i$.
Moreover,
we can define the edges of $G_i$ in a way that $A_i \cup S$ is adjacent
to $B_i \cup T$ if and only if $S, T \subset\{t+1, \cdots, n\}$, $S \cap
T = \emptyset$ and $|S|+|T| \ge k-t+1$. Note that by this definition,
two adjacent vertices cannot both correspond to nonnegative subsets,
otherwise $S \cup T \cup \{1, \cdots, t\}$ gives a nonnegative subset
of size larger than $k$. Applying Corollary \ref{perfect_matching}
with $m=n-t$, $r=k-t$, we conclude that there is a matching
saturating all
the vertices in $G_i$ except $A_i$ and $B_i$. Therefore the number of
nonnegative subsets in $G_i$ is at most $\binom{n-t}{k-t} + \cdots +
\binom{n-t}{0} + 1$. Note that this number remains the same for different
choices of $(A_i, B_i)$, so the total number of nonnegative subsets is
at most $2^{t-1}(\binom{n-t}{k-t} + \cdots + \binom{n-t}{0}+1)$.

\end{proof}

\section{A strengthening of the non-uniform EKR theorem}\label{section_alternative_proof}
A conjecture of
Manickam, Mikl\'{o}s, and Singhi (see \cite{MM}, \cite{MS})
asserts
that for any integers $n, k$ satisfying $n \geq 4k$, every set of
$n$ real
numbers with a nonnegative sum has at least $\binom{n-1}{k-1}$
$k$-element
subsets whose sum is also nonnegative. The study of this problem
(see, e.g., \cite{AHS} and the references therein) reveals a
tight connection between questions about nonnegative sums and
problems in extremal finite set theory.  A connection
of the same flavor
exists for the problem studied in this note, as explained in what
follows.

The Erd\H{o}s-Ko-Rado theorem \cite{erdos-ko-rado} has the following
non-uniform version: for integers $1 \le k \le n$, the maximum size
of an intersecting family of subsets of sizes up to $k$ is equal to
$\binom{n-1}{k-1}+ \binom{n-1}{k-2} + \cdots + \binom{n-1}{0}$. The
extremal example is the family of all the subsets of size at most $k$
containing a fixed element. This result is a direct corollary of
the uniform Erd\H{o}s-Ko-Rado theorem, together with the obvious fact
that
each such family cannot contain
a set and its complement. In this section we show that the following strengthening is also true. It also provides an alternative
proof of Theorem \ref{main_theorem}.

\begin{theorem} \label{s-ekr}
Let $1\leq k\leq n-1$, and let $\F \subset 2^{[n]}$ be a family consisting of subsets  of size at most $k$, where
$\emptyset\notin \F$. Suppose that for every two subsets $A,B\in \F$, if $A\cap B=\emptyset$, then $|A|+|B|\leq k$. Then
$|\F|\leq \binom{n-1}{k-1}+\binom{n-1}{k-2}+\ldots+\binom{n-1}{0}$.
\end{theorem}
\begin{proof}
Denote $\binom{[n]}{\leq k}=\{A\subset [n]: |A|\leq k\}$.  Let us first observe that if $\F$ is an upset in $\binom{[n]}{\leq k}$
(that is $A\in \F$ implies that $\{B\in \binom{[n]}{\leq k}: B\supseteq A\}\subset \F$)
then $\F$ is an intersecting family, and hence the bound for $|\F|$ holds. Suppose there exist $A,B\in \F$, such that $A\cap B=\emptyset$, thus $|A|+|B|\leq k$.
Since $k\leq n-1$ and $\F$ is an upset, there exists a $C\in \F$ such that $A\subset C$, $C\cap B=\emptyset$, and $|C|+|B|>k$ which is a contradiction.

Next let us show that applying so called ``pushing up'' operations $S_i(\F)$, we can transform  $\F$  to an upset $\F^*\subset\binom{[n]}{\leq k}$ of the same size, without
violating the property of $\F$. This, together with the observation above, will complete the proof.
For $i\in [n]$ we define $S_i(\F)=\{S_i(A): A\in \F\}$, where
 $$
S_i(A)=
  \begin{cases}
    A, & \hbox{if } A\cup\{i\}\in \F ~~ \hbox{or}~~ |A|=k \\
    A\cup\{i\}, & \hbox{otherwise.}
  \end{cases}
$$
 It is clear  that $|S_i(\F)|=|\F|$  and applying finitely
many operations $S_i(\F), i\in[n]$ we come to an upset $\F^*\subset \binom{[n]}{\leq k}$. To see that $S_i(\F)$ does not violate the property of $\F$ let $\F=\F_0\cup\F_1$, where $\F_1=\{A\in \F: i\in A\}, ~\F_0=\F\setminus\F_1$. Thus  $S_i(\F_1)=\F_1$. What we have to show is that  for each pair
 $A,B\in\F$  the pair $S_i(A), S_i(B)$ satisfies the condition in the theorem as well. 
In fact,  the only doubtful case is when  $A,B\in \F_0$, $A\cap B=\emptyset$, $|A|+|B|=k$. The subcase when $S_i(A)=A\cup \{i\}, S_i(B)=B\cup\{i\}$
is also clear. Thus, it remains to consider the situation
 when $S_i(A)=A$ (or $S_i(B)=B$). In this case $(A\cup\{i\})\in \F$,
since  $|A|,|B|\leq k-1$. Moreover, $(A\cup \{i\})\cap B=\emptyset$ and $|A\cup \{i\}|+|B|=k+1$,   a contradiction.
\end{proof}

To see that Theorem \ref{s-ekr} implies Theorem
\ref{main_theorem}, take $\mathcal{F}=\{F:
\emptyset \ne F \subset \{1, \cdots, n\}, \sum_{i \in F} x_i \ge 0
\}$. The family $\mathcal{F}$ satisfies the conditions in Theorem
\ref{s-ekr} since if $A, B \in \mathcal{F}$, then $\sum_{i
\in A} x_i \ge 0$, $\sum_{i \in B} x_i \ge 0$. If moreover $A \cap B
=\emptyset$, then $\sum_{i \in A \cup B} x_i \ge 0$ and it follows that
$|A \cup B| \le k$.

\section{Concluding remarks}
We have given two different proofs of the following result: for a
set of $n$ real numbers, if the sum of elements of every subset
of size larger than
$k$ is negative, then the number of subsets of nonnegative sum
is at most $\binom{n-1}{k-1}+ \cdots + \binom{n-1}{0}+1$.
The connection between questions of this type and extremal problems
for hypergraphs that appears here as well as in \cite{AHS} and some
of its
references is interesting and deserves further study.

Another intriguing question motivated by the first proof is the problem
of finding an explicit perfect matching
for Corollary \ref{perfect_matching}
without
resorting to Hall's Theorem. When $r$ is small or $r=m-1$, one can
construct such a perfect matchings, but it seems that things get
more complicated when $r$ is closer to $m/2$.
\vspace{0.2cm}

\noindent
{\bf Acknowledgment}\, We thank Emmanuel Tsukerman for telling us
about the problem considered here, and Benny Sudakov for fruitful
discussions, useful suggestions and helpful ideas.

\end{document}